\documentclass{amsart}

\RequirePackage{fix-cm}
\usepackage{graphicx}
\usepackage{amsmath, amsopn,amstext,amscd,amsfonts,amssymb}
\usepackage{dsfont}
\usepackage{comment}
\usepackage[active]{srcltx}
\usepackage{graphicx, epsfig, subfig}
\usepackage[dvipsnames]{xcolor,pstricks}
\usepackage{textcomp}

\usepackage{algorithm}

\makeatletter
\def\BState{\State\hskip-\ALG@thistlm}
\makeatother

\def\downbar#1{
\setbox10=\hbox{$#1$}
            \dimen10=\ht10 \advance\dimen10 by 2.5pt
            \ifdim \dimen10<15pt 
               \advance\dimen10 by -0.5pt
               \dimen11=\dimen10
               \advance\dimen10 by 2.5pt
               \lower \dimen11
            \else \lower \ht10 \fi
            \hbox {\hskip 1.5pt \vrule height \dimen10 depth \dp10}}
\def\upbar#1{
\setbox10=\hbox{$#1$}
            \dimen10=\ht10 \advance\dimen10 by \dp10 \advance\dimen10 by 2.5pt
            \ifdim \dimen10<15pt 
                \advance\dimen10 by 2pt \fi
            \raise 2.5pt \hbox {\hskip -1.5pt \vrule height \dimen10}}

\usepackage{multicol}
\usepackage{colortbl}
\usepackage{exscale}
\usepackage{amssymb,latexsym,amsthm,amsfonts,color,fancyhdr,mathrsfs}
\usepackage{lscape}
\usepackage{rotating}
\usepackage{caption}
\usepackage{hyperref}

\newtheorem{definition}{\bf Definition}[section]
\newtheorem{theorem}{\bf Theorem}[section]

\newtheorem{lemma}{\bf Lemma }[section]

\newtheorem{remark}{\bf Remark}[section]
\newtheorem{conj}{\bf Conjecture}[section]

\numberwithin{equation}{section}

\begin{document}

\title[On discrete coherent pairs of measures]
{On discrete coherent pairs of measures}
\author{R. \'Alvarez-Nodarse}
\address{University of Seville, Department of Mathematical Analysis, Spain}
\email{ran@us.es}
\author{K. Castillo}
\address{University of Coimbra, CMUC, Department of Mathematics, 3001-501 Coimbra, Portugal}
\email{kenier@mat.uc.pt}
\author{D. Mbouna}
\address{University of Coimbra, CMUC, Department of Mathematics, 3001-501 Coimbra, Portugal}
\email{dmbouna@mat.uc.pt}
\author{J. Petronilho}
\address{University of Coimbra, CMUC, Department of Mathematics, 3001-501 Coimbra, Portugal}

\date{\today}
\begin{abstract}
In [Castillo \& Mbouna, Indag. Math. {\bf 31} (2020) 223-234], the concept of $\pi_N$-coherent pairs of order $(m,k)$ with index $M$ is introduced. This definition, implicitly related with the standard derivative operator, automatically leaves out the so-called discrete orthogonal polynomials. The purpose of this note is  twofold: first we use the (discrete) Hahn difference operator and rewrite the known results in this framework; second, as an application, we describe exhaustively the (discrete) self-coherent pairs in the situation whether $M=0$, $N\leq2$, and $(m,k)=(1,0)$. 
This is proved by describing in a unified way the classical orthogonal polynomials with respect to Jackson's operator as special or limiting cases of a four parametric family of $q$-polynomials. This gives a partial answer to a conjecture posed by M. E. H Ismail in his monograph [Classical and quantum orthogonal polynomials in one variable, Cambridge University Press, 2005].
\end{abstract}
\maketitle

\section{Introduction and Preliminaries}

Following the ideas  presented in \cite{CD20}, this work provides an additional input to the theory of coherent pairs of measures in the framework of discrete orthogonal polynomials. The concept of coherent pairs of measure was introduced by Iserles, Koch, N{\o}rsett, Sanz-Serna in $1991$ in the context of the so-called Sobolev orthogonal polynomials and it was extensively studied after that (see, for instance, the introduction in \cite{APPR15,dJ14} and references therein). Our aim is not only to state appropriate $(q,w)$-analogues of the results presented in \cite{CD20}, but also to show how some known results can be easily derived using the ones presented here combined with some other ones stated in our recent work \cite{Nossa}.  
To achieve this we will introduce a four parametric family of $q$-polynomials that contains as a special or limiting case all the classical OPS with respect to Jackson's operator (i.e., the so-called $q$-classical polynomials).
We suppose that the reader has the papers \cite{Nossa, CD20} at hand and, as far as possible, we shall use the notation and definitions therein, including the basic facts of the algebraic theory of orthogonal polynomials stated by P. Maroni \cite{M91} (see also \cite{P18} for a recent survey on the subject).  Throughout this work we will use the abbreviations OP and OPS for Orthogonal Polynomial(s) and Orthogonal Polynomial(s) Sequence(s), respectively. Let $(P_n)_{n\geq0}$ and $(Q_n)_{n\geq0}$ be two monic OPS and let $\textbf{u}$ and $\textbf{v}$ be the regular linear functionals with respect to which $(P_n)_{n\geq0}$ and $(Q_n)_{n\geq0}$ are orthogonal. In \cite{CD20} the authors consider the following structure relation:
\begin{equation}\label{1ba}
\pi_N(x)P_n^ {[m]}(x) = \sum_{j=n-M}^{n+N} c_{n,j} Q_j ^{[k]}(x),\quad n=0,1,\ldots,  
\end{equation}
where $M$ and $N$ are fixed non-negative integer numbers, $\pi_N $ is a monic polynomial of degree $N$
(hence $c_{n,n+N}=1$ for each $n$), and $P_{n}^{[m]}(x):=\frac{m!}{(n-m)!}\frac{{\rm d}^m}{{\rm d}x^m}\,P_{n+m}(x)$ being the standard notation for the monic $m$-derivatives. We also consider the convention $Q_j\equiv0$ if $j<0$. 
Moreover, it is assume that the following conditions hold:
\begin{equation}\label{cond1}
c_{n,n-M} \neq 0\quad\mbox{\rm if}\quad n\geq M\;.
\end{equation}
Motivated by the preceding literature on the subject the following definition is stated in \cite{CD20}:
\begin{definition}\label{Def1}
Let $M$ and $N$ be non-negative integer numbers and let
$\pi_N$ be a monic polynomial of degree $N$. If $(P_n)_{n\geq0}$ and $(Q_n)_{n\geq0}$
are two monic OPS such that their normalized derivatives of orders $m$ and $k$ (respectively)
satisfy $(\ref{1ba})$ and $(\ref{cond1})$, we call $\big((P_n)_{n\geq0},(Q_n)_{n\geq0}\big)$,
as well as the corresponding pair $({\bf u},{\bf v})$ of regular linear functionals, a
$\pi_N$-coherent pair with index $M$ and order $(m,k)$.
\end{definition}

Let us define the so-called Hahn operator $D_{q,\omega}: \mathcal{P} \to \mathcal{P}$ by
\begin{align*}
(D_{q,\omega} f)(x) = \frac{f(qx+\omega)-f(x)}{(q-1)x+\omega}, \end{align*}
where we assume that $q,\omega \in \mathbb{C}$, and
$
q\not\in\{0,1,e^{2ij\pi/n}\;|\;1\leq j\leq n-1\;,\; n=2,3,\ldots\}.
$
Consider (\ref{1ba}) redefining the derivatives as ``discrete'' derivatives,
\begin{equation}\label{Snm}
S_{n}^{[m]}:=\frac{ [n]_q !}{[n+m]_q !} \; D_{q,\omega} ^m S_{n+m}.
\end{equation}
Here we use the standard notation
\begin{align*}
[0]_q ! :=1, \quad [n]_q!:=\prod_{j=1} ^{n} [j]_q, \quad [n]_q:=\frac{q^n -1}{q-1}, \quad  n=1,2,\dots.
\end{align*}
This leads to the concept of discrete $\pi_N$-$(q,w)$-coherent pair with index $M$ and order $(m,k)$,
defined as in Definition \ref{Def1} with the obvious modification; that is,
replacing in \eqref{1ba} the standard derivatives by the discrete ones \eqref{Snm}.

In the rest of the section we summarize some basic facts. 
(For more details see \cite{Nossa} and references therein.)
Given a simple set of polynomials $(R_n)_{n\geq0}$,
the corresponding dual basis is a sequence of linear functionals
${\bf e}_n:\mathcal{P}\to\mathbb{C}$ such that
$$
\langle{\bf e}_n,R_j\rangle:=\delta_{n,j},\quad n,j=0,1,\ldots.
$$
The operator ${D}_{q,\omega}$ induces an operator 
${\bf D}_{q,\omega}:\mathcal{P}^*\to\mathcal{P}^*$, given by
\begin{align*}
\langle{\bf D}_{q,\omega}{\bf u},f\rangle:=
-q^{-1}\langle{\bf u},D_{1/q,-\omega/q} f\rangle.
\end{align*}
The sequence $(R^{[k]} _n)_{n\geq 0}$ is a simple set of polynomials and its associated dual basis $(\textbf{e} ^{[k]}_n )_{n \geq 0}$ satisfies
\begin{align}\label{Der}
\mathbf{D}_{1/q,-\omega /q} ^{k} \left( \textbf{e}_n ^{[k]} \right) = (-q)^k \frac{[n+k]_q !}{[n]_q !}\textbf{e}_{n+k}, \quad n,k=0,1,\dots.
\end{align}
We also define
$L_{q,\omega}:\mathcal{P}\to\mathcal{P}$ and ${\bf L}_{q,\omega}:\mathcal{P}^*\to\mathcal{P}^*$ by
\begin{align*}
L_{q,\omega}f(x):=f(qx+\omega),\quad
\langle {\bf L}_{q,\omega}{\bf u},f\rangle:=\langle {\bf u},L_{1/q,-\omega/q}f\rangle.
\end{align*}
Throughout the paper the following properties will be used
\begin{align*}
& D_{1/q,-\omega/q} L_{q,\omega}=q L_{q,\omega} D_{q,\omega},\quad  L_{1/q,-\omega/q}L_{q,\omega}=I, \\[7pt]
& {\bf D}_{q,\omega} (f \textbf{u})=D_{q,\omega}f ~\textbf{u} + L_{q,\omega}f ~ {\bf D}_{q,\omega}\textbf{u}
\;,\quad {\bf L}_{q,\omega} (f\textbf{u})=L_{q,\omega}f {\bf L}_{q,\omega}(\textbf{u}),\\[7pt]
& {\bf D}_{q,\omega} ^n (f\textbf{u})=\sum_{j=0} ^n \begin{bmatrix}
n \\
j
\end{bmatrix} _q L_{q,\omega} ^{n-j} \left(D_{q,\omega} ^{j} f  \right){\bf D}_{q,\omega}^{n-j} \textbf{u}
=\sum_{j=0} ^n \begin{bmatrix}
n \\
j
\end{bmatrix} _q L_{q,\omega} ^{j} \left(D_{q,\omega} ^{n-j} f  \right){\bf D}_{q,\omega}^{j} \textbf{u}
\end{align*}
for each $n=0,1,\ldots$. The last property is {\it Leibniz's formula}.

A linear regular functional $\textbf{u} \in \mathcal{P}^*$ is a $D_{q,\omega}$-semiclassical functional if it is regular and there exist $\phi ,\psi \in \mathcal{P}$, with $deg ~\psi \geq 1$, such that
\begin{align}\label{DPearson}
{\bf D}_{q,\omega} \left( \phi \textbf{u}\right) =\psi \textbf{u}.
\end{align}
The class of a $D_{q,\omega}$-semiclassical functional $\textbf{u}$, denoted by $\mathfrak{s}({\bf u})$, is the unique non-negative integer number defined by
$$
\mathfrak{s}({\bf u}):=\min_{(\phi,\psi)\in\mathcal{A}_{\bf u}}\max\big\{\deg\phi-2,\deg\psi-1\big\}\;,
$$
where $\mathcal{A}_{\bf u}$ is the set of all pairs of nonzero polynomials $(\phi,\psi)$ fulfilling the functional equation (\ref{DPearson}). When $\mathfrak{s}({\bf u})=0$, $\bf u$ is said to be a $D_{q,\omega}$-classical functional. We also say that the corresponding OPS is $D_{q,\omega}$-semiclassical of class $\mathfrak{s}({\bf u})$ or $D_{q,\omega}$-classical, respectively.
Notice that
\begin{align*}
{\bf D}_{q,\omega} (\phi \textbf{u}) =\psi \textbf{u} ~ \Leftrightarrow ~ {\bf D}_{1/q,-\omega/q} ({\widehat{\phi}} \textbf{u}) =\psi \textbf{u},
\end{align*}
where $\widehat{\phi}(x):=q^{-1} \left[ \phi (x)+\left((q-1)x+\omega  \right)\psi (x) \right]$. Consequently, a regular linear functional $\textbf{u}$ is  $D_{q,\omega}$-semiclassical if and only if it is  $D_{1/q,-\omega/q}$-semiclassical.

The structure of the paper is as follows.
In Section \ref{Section-main} we state the discrete analogues of the results presented in \cite{CD20}. These results are new in the literature. 
In Section \ref{Section-example} we present a novel application in the framework of $\pi_N$-$(q,w)$-coherence
of index $0$ and order $(1,0)$, being $N\leq2$ and considering $P_n=Q_n$ for each $n=0,1,2,\dots$ (self-coherence). This idea is new and unprecedented. This allow us to prove a $(q,\omega)$-analogue of the well knwon Al-Salam and Chihara characterization of classical orthogonal polynomials. 
We also describe in a unified way (up to an affine change of variable) all the $q$-classical OPS, by splitting them into two large families of $q$-polynomials, $(L_n(\cdot;a,b,c|q))_{n\geq0}$ and $(J_n(\cdot;a,b,c,d|q))_{n\geq0}$. As a matter of fact, even $L_n(\cdot;a,b,c|q)$ can be obtained as a special or limiting case of $J_n(\cdot;a,b,c,d|q)$.  

\section{$(q,w)$-analogues of $\pi_N$-coherent pairs}\label{Section-main}

In this section we built the $(q,w)$-analogues of the results obtained in \cite{CD20}.
For doing that we follow the ideas presented in \cite{CD20}.
But first, and this is a delicate point, we need to define the appropriate $(q,w)$-analogues of the formulas appearing therein.\footnote{Indeed, as B. Berndt emphasized in a slightly different context (see \cite[Section 4]{B2010}), there are several ways to construct the so-called $q$-analogues and, \textit{a fortiori}, the $(q,w)$-analogues.}
After that, it is straightforward to prove our results following the arguments used in \cite{CD20}.

\begin{lemma}\label{lemma-main}
Let $\big((P_n)_{n\geq0},(Q_n)_{n\geq0}\big)$ be a $\pi_N$-$(q,\omega)$-coherent pair with index $M$ and order $(m,k)$,
so that $(\ref{1ba})$ and $(\ref{cond1})$ hold. Set
\begin{align}
\psi(x;n)&:=\sum_{j=n-N}^{n+M}   \frac{(-q)^m [j+m]_q ! }{[j]_q ! \langle {\bf u}, P_{m+j} ^2\rangle} c_{j,n} P_{m+j}(x), \label{si}\\[7pt]
 \label{fi}\phi(x;n,j)&:= \frac{(-q)^k [n+k]_q !}{[n]_q ! \langle{\bf v},Q_{n+k}^2\rangle}\sum_{\ell=0}^{N-j} \begin{bmatrix}
k+n \\
\ell
\end{bmatrix}_{q^{-1}} \begin{bmatrix}
N-\ell \\
N-j-\ell
\end{bmatrix} _{q^{-1}} \\[7pt]
&\nonumber \quad \times L_{1/q,-\omega/q} ^{ k+N-\ell} \left(D_{1/q,-\omega/q} ^{\ell} \pi_N  \right)(x)
\,L_{1/q,-\omega/q} ^{ j} \left(D_{1/q,-\omega/q} ^{ n-j-\ell} Q_{n+k}  \right)(x),
\end{align}
for all $\,n=0,1,\ldots$ and $j=0,1,\ldots,N$, so that
$$\deg\psi(\cdot;n)=m+n+M\;,\quad\deg\phi(\cdot;n,j)=k+n+j\;.$$
Let ${\bf u}$ and ${\bf v}$ be the regular linear functionals
with respect to which $(P_n)_{n\geq0}$ and $(Q_n)_{n\geq0}$ are orthogonal.
Then the following functional equations hold:
\begin{align}
& \psi(\cdot;n){\bf u}={\bf D}_{1/q,-\omega/q} ^{ m-k-N}\Big(\sum_{j=0}^{N}\phi(\cdot;n,j){\bf D}_{1/q,-\omega/q}^{j}{\bf v}\Big)
& \mbox{\rm if}\quad m \geq k+N, \label{2a} \\
& {\bf D}_{1/q,-\omega/q} ^{ k+N-m}\big(\psi(\cdot;n){\bf u}\big)=\sum_{j=0}^{N}\phi(\cdot;n,j){\bf D}_{1/q,-\omega/q}^{j} {\bf v}
& \mbox{\rm if}\quad m< k+N, \label{2b}
\end{align}
for all $n=0,1,\ldots$.
\end{lemma}

\begin{proof}
Let $({\bf a}_n)_{n\geq0}$, $({\bf b}_n)_{n\geq0}$, $({\bf a}_n^{[m]})_{n\geq0}$,
and $({\bf b}_n^{[k]})_{n\geq0}$ be the dual basis corresponding to the simple sets of polynomials
$(P_n)_{n\geq0}$, $(Q_n)_{n\geq0}$, $(P_n^{[m]})_{n\geq0}$ and $(Q_n^{[k]})_{n\geq0}$, respectively. Then
\begin{equation}\label{rr}
\pi_N {\bf b}_n ^{[k]} = \sum_{j=n-N} ^{n+M} c_{j,n} {\bf a}_j ^{[m]}, \quad n =0,1,\ldots.
\end{equation}
From \eqref{Der} we get
\begin{equation}\label{basis}
{\bf D}_{1/q,-\omega/q} ^{m} \big(\pi_N {\bf b}_n ^{[k]}  \big) = \psi(\cdot;n)\textbf{u}, \quad n=0,1,\ldots.
\end{equation}
By Leibniz's formula, and since $D_{1/q,-\omega/q}^j \pi_N =0$ for $j>N$,  we deduce
\begin{align*}
&{\bf D}_{1/q,-\omega/q}^{ k+N} \big(\pi_N {\bf b}_n ^{[k]} \big)\\[7pt]
&=\frac{(-q)^k [n+k]_q !}{[n]_q ! \langle {\bf v}, Q_{n+k}^2 \rangle} \sum_{j=0}^{N}
\begin{bmatrix}
k+N \\
j
\end{bmatrix}_{q^{-1}}
L_{1/q,-\omega/q} ^{ k+N-j} \left(D_{1/q,-\omega/q} ^{ j} \pi_N \right) {\bf D}_{1/q,-\omega/q} ^{ N-j}(Q_{n+k} {\bf v}).
\end{align*}
Applying once again Leibniz's formula to ${\bf D}_{1/q,-\omega/q} ^{ N-j}(Q_{n+k} {\bf v})$, after some straightforward calculations, we find
\begin{equation}\label{EqDkN1}
{\bf D}_{1/q,-\omega/q} ^{k+N} \big(\pi_N {\bf b}_n ^{[k]} \big)
= \sum_{j=0}^{N}  \phi(\cdot;n,j) {\bf D}^{j}_{1/q,-\omega/q} {\bf v}.
\end{equation}
For $m \geq k+N$, rewriting  \eqref{basis} as
$$
\psi(\cdot;n){\bf u} ={\bf D}_{1/q,-\omega/q} ^{ m-k-N } {\bf D}_{1/q,-\omega/q} ^{ k+N} \big(\pi_N {\bf b}_n ^{[k]}  \big),\quad n=0,1,\ldots,
$$
and using (\ref{EqDkN1}), (\ref{2a}) follows.
For $m< k+N$, writing
$$
{\bf D}_{1/q,-\omega/q} ^{ k+N} \big(\pi_N {\bf b}_n ^{[k]} \big)={\bf D}_{1/q,-\omega/q} ^{ k-m+N } {\bf D}_{q,\omega} ^{m} \big(\pi_N {\bf b}_n ^{[k]} \big)\;,\quad n=0,1,\ldots,
$$
and using (\ref{basis}) and (\ref{EqDkN1}), we obtain (\ref{2b}).
\end{proof}

\begin{theorem}\label{theorem-main}{\rm (Case $m\geq k+N$)}
Let $\big((P_n)_{n\geq0},(Q_n)_{n\geq0}\big)$ be a $\pi_N$-$(q, \omega)$-coherent pair with index $M$ and order $(m,k)$,
so that $(\ref{1ba})$ and $(\ref{cond1})$ hold.
Let ${\bf u}$ and ${\bf v}$ be the regular linear functionals with respect to which
$(P_n)_{n\geq0}$ and $(Q_n)_{n\geq0}$ are orthogonal.
Suppose $m \geq k+N$. Assume further that $m>k$ whenever $N=0$.
For each $i=0,\ldots,m-k$ and $n=0,1,\ldots$, let
\begin{equation}\label{EqVarphi1}
\varphi(x;n,i):=\sum_{\substack{j+\ell=i \\ 0\leq j\leq N \\ 0\leq \ell\leq M}}
\begin{bmatrix}
m-k-N \\
j
\end{bmatrix}_{q^{-1}}
    L_{1/q,-\omega/q} ^{j} \big( D_{1/q,-\omega/q} ^{m-k-N-j}\phi(.;n,\ell)\big)(x)\;,
\end{equation}
$\phi(\cdot;n,j)$ being the polynomial introduced in $(\ref{fi})$.
Let $\mathcal{A}(x)$ be the polynomial matrix of order $m-k+1$ defined by
$$
\mathcal{A}(x):=\big[\varphi(x;n,j)\big]_{n,j=0}^{m-k}\;.
$$
Let $\mathcal{A}_1(x)$ (resp., $\mathcal{A}_2(x)$) be the matrix obtained
by replacing the first (resp., the second) column of
$\mathcal{A}(x)$ by $\big[\psi(x;0),\psi(x;1),\ldots,\psi(x;m-k)]^t$,
and set
$$
A(x):=\det\mathcal{A}(x)\;,\quad A_1(x):=\det\mathcal{A}_1(x)\;,\quad A_2(x):=\det\mathcal{A}_2(x)\;.
$$
Assume that the polynomial $A(x)$ does not vanish identically.
Then
\begin{equation}\label{EqAvA1u}
A{\bf v}=A_1{\bf u}\;,\quad A {\bf D}_{1/q,-\omega/q}  {\bf v}  = A_2{\bf u}\;,
\end{equation}
hence ${\bf u}$ and ${\bf v}$ are $D_{q,\omega}$-semiclassical functionals related by a rational transformation.
Moreover, ${\bf u}$ and ${\bf v}$ fulfill the following equations:
\begin{align}
{\bf D}_{1/q,-\omega/q}  (A_1 L_{q,\omega} (A ){\bf u})&=\big( qA_1 D_{q,\omega}(A)  +A_1 D_{1/q,-\omega/q}  (A)   +A_2 L_{1/q,-\omega/q}(A) \big) {\bf u}, \label{FucEquv1}\\[7pt]%
{\bf D}_{1/q,-\omega/q}  (L_{q,\omega} (AA_1 ){\bf v})&=\big( qD_{q,\omega} (AA_1)  +A A_2 \big) {\bf v}  \label{FucEquv2}.
\end{align}
\end{theorem}

\begin{proof}
Combining (\ref{2a}) and Leibniz formula, we get
$$
\psi(\cdot;n){\bf u}
=\sum_{\ell=0}^{N}\sum_{j=0}^{m-k-N} 
\begin{bmatrix}
m-k-N \\
j
\end{bmatrix}_{q^{-1}} 
L_{1/q,-\omega/q} ^{j} \big( D_{1/q,-\omega/q} ^{m-k-N-j} \phi(\cdot;n,\ell)\big)  {\bf D}_{1/q,-\omega/q} ^{j+\ell} {\bf v}.
$$
This may be rewritten as
\begin{equation}\label{varphi-uv}
\psi(\cdot;n){\bf u}=\sum_{\ell=0}^{m-k}\varphi(\cdot;n,\ell){\bf D}_{1/q,-\omega/q} ^{\ell} {\bf v}
,\quad n=0,1,\dots,
\end{equation}
where $\varphi(\cdot;n,i)$ is the polynomial introduced in (\ref{EqVarphi1}).
Taking $n=0,1,\ldots,m-k$ in (\ref{varphi-uv}) we obtain a system with $m-k+1$ equations that can be written as
$$
\left(
\begin{array}{c}
\psi(x;0){\bf u} \\
\psi(x;1){\bf u} \\
\vdots \\
\psi(x;m-k){\bf u}
\end{array}
\right)
=\mathcal{A}(x)
\left(
\begin{array}{c}
{\bf v} \\
{\bf D}_{1/q,-\omega/q}{\bf v} \\
\vdots \\
{\bf D}_{1/q,-\omega/q} ^{ m-k}{\bf v}
\end{array}
\right).
$$
Solving for ${\bf v}$ and ${\bf D}_{1/q,-\omega/q} {\bf v}$ we obtain (\ref{EqAvA1u}).
Finally, one can remark that $D_{1/q,-\omega/q} L_{q,\omega} = qD_{q,\omega}$. Hence (\ref{FucEquv1}) and (\ref{FucEquv2}) follow from (\ref{EqAvA1u}).  
\end{proof}

\begin{remark}
As it is commented in \cite{CD20}, if $m=k$ and $N=0$, then ${\bf u}$ and ${\bf v}$ are still related by a rational transformation, but we cannot ensure that they are  $D_{q,\omega}$-semiclassical.
\end{remark}

\begin{theorem}\label{theorem-mainB}{\rm (Case $m<k+N$)}
Let $\big((P_n)_{n\geq0},(Q_n)_{n\geq0}\big)$ be a $\pi_N$-$(q,\omega)$-coherent pair with index $M$ and order $(m,k)$,
so that $(\ref{1ba})$ and $(\ref{cond1})$ hold.
Let ${\bf u}$ and ${\bf v}$ be the regular linear functionals with respect to which
$(P_n)_{n\geq0}$ and $(Q_n)_{n\geq0}$ are orthogonal.
Assume further that $m < k+N$.
For each $j=0,\ldots,k-m+N$ and $n=0,1,\ldots$, set
\begin{equation}\label{Eqxi1}
\xi(x;n,j):=\begin{bmatrix}
k+N-m\\
j
\end{bmatrix}_{q^{-1}} L_{1/q,-\omega/q} ^{j} \big(D_{1/q,-\omega/q} ^{k+N-m-j}  \psi \big)(x;n) \;,
\end{equation}
$\psi(\cdot;n)$ being the polynomial introduced in $(\ref{si})$.
Let $\mathcal{B}(x):= \big[b_{i,j}(x)\big]_{i,j=0}^{k-m+2N}$ be the polynomial matrix of order $k-m+2N+1$ defined by
$$
b_{i,j}(x):=\left\{
\begin{array}{ccl}
\phi(x;i,j) & \mbox{\rm if} & 0\leq j\leq N\;, \\ [0.5em]
-\xi(x;i,j-N) & \mbox{\rm if} & N+1\leq j\leq k-m+2N\;,
\end{array}
\right.
$$
$\phi(\cdot;i,j)$ being the polynomial given by $(\ref{fi})$.
Let $\mathcal{B}_1(x)$ (resp., $\mathcal{B}_2(x)$ and $\mathcal{B}_{N+2}(x)$) be the matrix obtained
by replacing the first (resp., the second and the $(N+2)$-th) column of
$\mathcal{B}(x)$ by $\big[\xi(x;0,0),\xi(x;1,0),\ldots,\xi(x;m-k+2N,0)]^t$,
and set
$$
B(x):=\det\mathcal{B}(x)\;,\quad B_j(x):=\det\mathcal{B}_j(x)\;,\quad j\in\{1,2,N+2\}\;.
$$
Assume that the polynomial $B(x)$ does not vanish identically.
Then
\begin{equation}\label{EqAvB1u}
B{\bf v}=B_1{\bf u}\;,\quad B{\bf D}_{1/q,-\omega/q}  {\bf v}=B_2{\bf u}\;,\quad B{\bf D}_{1/q,-\omega/q}  {\bf u}=B_{N+2}{\bf u}\;,
\end{equation}
hence ${\bf u}$ and ${\bf v}$ are  $D_{q,\omega}$-semiclassical functionals related by a rational transformation.
Moreover, ${\bf u}$ and ${\bf v}$ fulfill the following equations:
\begin{align}\label{FucEquvB}
&{\bf D}_{1/q,-\omega/q}  ( L_{q,\omega} (BB_1){\bf v})=\big(q D_{q,\omega}(BB_1) +BB_2   \big){\bf v} \\
&{\bf D}_{1/q,-\omega/q}  (L_{q,\omega}B {\bf u})=\big(q D_{q,\omega}B +B_{N+2} \big) {\bf u}.
\end{align}
\end{theorem}

\begin{proof}
Using Leibniz formula we can rewrite (\ref{2a}) as
$$
\sum_{j=0}^{k-m+N}\xi(\cdot;n,j){\bf D}_{1/q,-\omega/q} ^{j}{\bf u}
=\sum_{j=0}^{N}\phi(\cdot;n,j){\bf D}_{1/q,-\omega/q} ^{j}{\bf v},\quad n=0,1,\ldots.
$$
Taking $n=0,1,\ldots,k-m+2N$, we obtain the following system of $k-m+2N+1$ equations:
$$
\left(
\begin{array}{c}
\xi(x;0,0){\bf u} \\
\xi(x;1,0){\bf u} \\
\vdots \\
\xi(x;k-m+N,0){\bf u} \\
\xi(x;k-m+N+1,0){\bf u} \\
\vdots \\
\xi(x;k-m+2N,0){\bf u}
\end{array}
\right)
=\mathcal{B}(x)
\left(
\begin{array}{c}
{\bf v} \\
{\bf D}_{1/q,-\omega/q}  {\bf v} \\
\vdots \\ [0.25em]
{\bf D}_{1/q,-\omega/q}^{N} {\bf v} \\
{\bf D}_{1/q,-\omega/q} {\bf u} \\
\vdots \\ [0.25em]
{\bf D}_{1/q,-\omega/q} ^{k-m+N}{\bf u}
\end{array}
\right)\;.
$$
The theorem follows by solving this system for ${\bf v}$, ${\bf D}_{1/q,-\omega/q}  {\bf v}$, and ${\bf D}_{1/q,-\omega/q} {\bf u}$.
\end{proof}

\begin{theorem}\label{theorem-main-kzero}{\rm (Case $k=0$)}
Let $\big((P_n)_{n\geq0},(Q_n)_{n\geq0}\big)$ be a $\pi_N$- $(q,\omega)$-coherent pair with index $M$ and order $(m,0)$,
so that the structure relation
$$
\pi_N(x)P_n^ {[m]}(x) = \sum_{j=n-M}^{n+N} c_{n,j} Q_j(x),\quad n=0,1,\ldots
$$
holds, where $M$ and $N$ are fixed non-negative integer numbers, $\pi_N $ is a monic polynomial of degree $N$,
and $c_{n,n-M} \neq 0$ if $n\geq M$.
Assume further that $m\geq1$ if $N=0$.
Let ${\bf u}$ and ${\bf v}$ be the regular linear functionals with respect to which
$(P_n)_{n\geq0}$ and $(Q_n)_{n\geq0}$ are orthogonal.
Then ${\bf u}$ and ${\bf v}$ are  $D_{q,\omega}$-semiclassical functionals related by a rational transformation. More precisely, setting
\begin{equation}\label{EqPhikzero1}
\Phi(x;j):=\frac{\langle{\bf v},Q_j^2\rangle\psi(x;j)- \displaystyle \sum_{ l=0 } ^{j-1} \begin{bmatrix}
m\\
l
\end{bmatrix} _{q^{-1}}
L_{1/q,-\omega/q} ^{m-l}\big(D_{1/q,-\omega/q} ^{l}  Q_j \big)(x) \Phi(x;l)}{[j]_{q^{-1}}!  \begin{bmatrix}
m \\
j
\end{bmatrix} _{q^{-1}} },
\end{equation}
$j=0,1,\ldots$, $\psi(\cdot;j)$ being the polynomial introduced in $(\ref{si})$,
then $\deg\Phi(\cdot;0)=M+m$, $\deg\Phi(\cdot;j)\leq M+m+j$ for each $j=1,\ldots,m$,
and the following holds:
\begin{align}
{\bf D}_{1/q,-\omega/q}  \big(\Phi(\cdot;1){\bf u}\big)&=\Phi(\cdot;0){\bf u}, \label{EqFuc1}\\[7pt]
\pi_N{\bf v}&=\Phi(\cdot;m){\bf u}, \label{EqFuc2}\\[7pt]
{\bf D}_{1/q,-\omega/q}  \big(L_{q,\omega}\Phi(\cdot;m)) \pi_N{\bf v}\big)&= \left( qD_{q,\omega} \Phi(\cdot;m)+\Phi(\cdot;m-1)\right) \pi_N{\bf v}. \label{EqFuc3}
\end{align}
Moreover, $\mathfrak{s}({\bf u})\leq M+m-1$ and $\;\mathfrak{s}({\bf v})\leq N+M+2(m-1)$.
\end{theorem}

\begin{proof}
Since $k=0$, then ${\bf b}_n^{[k]}\equiv{\bf b}_n^{[0]}={\bf b}_n=\langle {\bf v},Q_n^2\rangle^{-1}Q_n{\bf v}$ for each $n=0,1,\ldots$,
hence relation (\ref{basis}) may be rewritten as
\begin{equation}\label{basis-kzero1}
{\bf D}_{1/q,-\omega/q} ^{m} \big(Q_n \pi_N {\bf v}\big) = \langle {\bf v},Q_n^2\rangle\psi(\cdot;n)\textbf{u}, \quad n=0,1,\ldots,
\end{equation}
where $\psi(\cdot;n)$ is defined by (\ref{si}).
Taking $n=0$, we obtain
\begin{equation}\label{basis-kzero2}
{\bf D}_{1/q,-\omega/q} ^{m} \big(\pi_N {\bf v}\big) = \Phi(\cdot;0)\textbf{u}.
\end{equation}
Taking $n=1$ in (\ref{basis-kzero1}) and then applying Leibniz's formula, we deduce
\begin{align*}
\langle {\bf v},Q_1^2\rangle\psi(\cdot;1)\textbf{u}&=
{\bf D}_{1/q,-\omega/q} ^{m} \big(Q_1 \pi_N {\bf v}\big)\\
&= [m]_{q^{-1}}{\bf D}_{1/q,-\omega/q} ^{m-1} \big(\pi_N {\bf v}\big)+ L_{1/q,-\omega/q} ^{m} Q_1 {\bf D}_{1/q,-\omega/q} ^{m} \big(\pi_N {\bf v}\big).
\end{align*}
Hence, by (\ref{basis-kzero2}), we have
\begin{equation}\label{basis-kzero3}
{\bf D}_{1/q,-\omega/q} ^{m-1} \big(\pi_N {\bf v}\big) = \Phi(\cdot;1)\textbf{u}.
\end{equation}
Thus (\ref{EqFuc1}) follows from (\ref{basis-kzero2}) and (\ref{basis-kzero3}).
This proves that ${\bf u}$ is  $D_{q,\omega}$-semiclassical of class $\mathfrak{s}({\bf u})\leq M+m-1$.
We conclude pursuing with the described procedure,
so that by taking successively $n=0,1,\ldots,m$ in (\ref{basis-kzero1}), the following relation holds:
\begin{equation}\label{basis-kzero4}
{\bf D}_{1/q,-\omega/q} ^{m-j} \big(\pi_N {\bf v}\big) = \Phi(\cdot;j)\textbf{u},\quad j=0,1,\ldots,m.
\end{equation}
In particular, for $j=m$ we obtain (\ref{EqFuc2}), hence ${\bf u}$ and ${\bf v}$
are related by a rational transformation. Setting $j=m-1$ in (\ref{basis-kzero4}),
we obtain
\begin{equation}\label{basis-kzero5}
{\bf D}_{1/q,-\omega/q} \big(\pi_N {\bf v}\big) = \Phi(\cdot;m-1)\textbf{u}.
\end{equation}
Since
${\bf D}_{1/q,-\omega/q} \big( L_{q,\omega}(\Phi(\cdot;m))\pi_N{\bf v}\big)=qD_{q,\omega} (\Phi(\cdot;m))\pi_N{\bf v}+ \Phi(\cdot;m) {\bf D}_{1/q,-\omega/q} \big(\pi_N{\bf v}\big)$,
we obtain (\ref{EqFuc3}) using (\ref{basis-kzero5}) and (\ref{EqFuc2}).
Thus ${\bf v}$ is  $D_{q,\omega}$-semiclassical of class $\mathfrak{s}({\bf v})\leq N+M+2m-2$,
and the theorem is proved.
\end{proof}

\section{An Application}\label{Section-example}

The interest of the results presented in the previous section
will be illustrated by an exhaustive analysis of the $\pi_N$-$(q,w)$-coherent pairs of index $M=0$ and order $(m,k)=(1,0)$, considering $N\leq2$ and $P_n=Q_n$ for each $n=0,1,\ldots$.
This means that we focus on the structure relation
\begin{equation}\label{Griffin1}
\pi_N(x)  D_{q,\omega}P_{n+1}(x)=[n+1]_q \sum_{j=n} ^{n+N} c_{n,j} P_{j} (x), \quad n=0,1,\ldots\;,
\end{equation}
where $\pi_N$ is a monic polynomial of degree $N\in\{0,1,2\}$
and the $c_{n,j}$ are complex numbers subject to the conditions $c_{n,n+N}=1$ and $c_{n,n}\neq0$ for each $n=0,1,2,\ldots$.
Our aim is to describe all the monic OPS $(P_n)_{n\geq0}$ fulfilling \eqref{Griffin1}. 
We prove in a rather simple and constructive way that, up to affine transformations of variable, the only monic OPS satisfying \eqref{Griffin1} are the monic $q$-classical OPS given in Table \ref{table1}.
This is a $(q,\omega)$-analogue of the well known characterization of classical OPS (Hermite, Laguerre, Jacobi, and Bessel) due to Al-Salam and Chihara \cite{ASC1972}.

\begin{table}[htp]
\caption{Monic $q$-classical OPS}
\begin{center}
\begin{tabular}{|c|c|c|c|}
\hline
Name &  Notation ($P_n$) & Restrictions & Reference \\
\hline\hline
Al-Salam-Carlitz & $U_n^{(a)}(\cdot|q)$ & $a\neq0$ & \cite[(14.24.4)]{KSL2010} \\
\hline
Big-$q$-Laguerre & $L_n(\cdot;a,b|q)$ & $ab\neq0$\;;\; $a,b\not\in\Lambda$ & \cite[(14.11.4)]{KSL2010} \\
\hline
Little-$q$-Laguerre & $L_n(\cdot;a|q)$ & $a\neq0$\;;\; $a\not\in\Lambda$ & \cite[(14.20.4)]{KSL2010} \\
\hline
------ & $l_n(\cdot;a|q)$ & $a\neq0$ & \cite[Table 2]{MA-N2001} \\
\hline
Big-$q$-Jacobi & $P_n(\cdot;a,b,c|q)$ & $ac\neq0$\;;\; $a, b,c,ab,abc^{-1}\not\in\Lambda$  & \cite[(14.5.4)]{KSL2010} \\
\hline
Little-$q$-Jacobi & $P_n(\cdot;a,b|q)$ & $a\neq0$\;;\; $a,b,ab\not\in\Lambda$ & \cite[(14.12.4)]{KSL2010} \\
\hline
$q$-Bessel & $B_n(\cdot;a|q)$  & $a\neq0$\;;\; $-a\not\in\Lambda$ & \cite[(14.22.4)]{KSL2010} \\
\hline
------ & $j_n(\cdot;a,b|q)$  & $ab\neq0$\;;\; $a\not\in\Lambda$ & \cite[Table 2]{MA-N2001} \\
\hline
\end{tabular}
\end{center}
\label{table1}
\end{table}

\noindent
The set $\Lambda$ in Table 1 is defined by $\Lambda:=\{q^{-n}\,:\,n=1,2,\ldots\}$.
The three-term recurrence relation characterizing the monic OPS $(P_n)_{n\geq0}$ will be written as
\begin{equation}\label{TTRRPn}
xP_n(x)=P_{n+1}(x)+\beta_n P_n(x)+\gamma_nP_{n-1}(x),\quad n=0,1,\ldots \;,
\end{equation}
$(\beta_n)_{n\geq0}$ and $(\gamma_n)_{n\geq1}$ being
sequences of complex numbers such that $\gamma_n \neq 0$ for each $n=1,2,\ldots$, and
$P_{-1}(x)=0$. It is well known that 
\begin{equation}\label{TTRRPnCo}
\beta_n=\frac{\langle{\bf u},xP_n^2\rangle}{\langle{\bf u},P_n^2\rangle}\;,\quad
\gamma_{n+1}=\frac{\langle{\bf u},P_{n+1}^2\rangle}{\langle{\bf u},P_n^2\rangle}
,\quad n=0,1,\ldots \,,
\end{equation}
where ${\bf u}$ is the regular linear functional with respect to which $(P_n)_{n\geq0}$ is a monic OPS.
We will show that the possible families $(P_n)_{n\geq0}$ fulfilling \eqref{Griffin1} may be related 
(up to affine changes of the variable) to one of the following two OPS:
\medskip

{\bf (I)} The monic OPS $(L_n (x;a,b,c|q))_{n\geq 0}$ given by \eqref{TTRRPn}, where  
\begin{align*}
\beta_n &= \left( a+b-c(q^{n+1} +q^n -1)\right)q^n \;,\\
\gamma_{n+1}  &=- \left( a-cq^{n+1} \right)(b-cq^{n+1})(1-q^{n+1})q^n
\end{align*}
for each $n=0,1,2,\ldots$, and $a,b,c \in \mathbb{C}$ are parameters subject to the regularity conditions
$$ a\neq cq^n\;,\quad b \neq cq^n $$ 
for each $n=1,2,\ldots$.
Although there are three parameters in the definition of $L_n (x;a,b,c|q)$, we note that, without loss of generality, if $c\neq0$ then,  
up to an affine change of variables, we may reduce to the case $c=1$. Indeed, the relation
$$L_n (x;a,b,c|q)= c^{n}\,L_n(x/c;a/c,b/c,1|q)$$
holds for each $n=0,1,2,\ldots$. Moreover, if $c=0$ (and so $ab\neq0$, by the regularity conditions), then 
up to the affine change of variable $x\mapsto bx$, we may reduce to the case $b=1$, taking into account that the relation
$$L_n (x;a,b,0|q)=b^{n}\,L_n(x/b;a/b,1,0|q)$$
holds for each $n=0,1,2,\ldots$.
\medskip

{\bf (II)} The monic OPS $(J_n (x;a,b,c,d|q))_{n\geq 0}$ given by (\ref{TTRRPn}), where
\begin{align*}
\beta_n &= q^n\,\frac{[a(b+d)+c(b+1)](1+dq^{2n+1})-[c(b+d)+ad(b+1)](1+q)q^n}{(1-dq^{2n})(1-dq^{2n+2})}  \,,\\ 
 \gamma_{n+1} &=-\frac{q^n(1-q^{n+1})(1-bq^{n+1})(1-dq^{n+1})(a-cq^{n+1})(b-dq^{n+1})(c-adq^{n+1})}{(1-dq^{2n+1})(1-dq^{2n+2})^2 (1-dq^{2n+3})}
\end{align*}
for each $n=0,1,2,\ldots$, where $a,b,c,d \in \mathbb{C}$ fulfill the regularity conditions
$$ b \neq q^{-n}\;,\quad d\neq q^{-n}\;,\quad a\neq cq^{n}\;,\quad b\neq dq^{n}\;,\quad c\neq adq^n$$
for each $n=1,2,\ldots$.

\begin{remark}
Note that the monic OPS given in {\bf (I)} is a special or limiting case of the monic OPS given in {\bf (II)}. 
Indeed, the following relations hold:
\begin{align}
L_n\big(x;a,b,c| q\big)&=
\left\{
\begin{array}{rcl}
J_n(x;ab/c,c/b,b,0|q) & \mbox{\rm if} & bc\neq0\,, \\ [0.25em]
J_n(x;ab/c,c/a,a,0|q) & \mbox{\rm if} & ac\neq0\;,
\end{array}
\right.  \nonumber \\
L_n(x;0,0,c|q) &=\lim_{b\to0} J_n(x;0,c/b,b,0|q)\quad
\mbox{\rm if}\quad  c\neq0 \;, \nonumber \\  
L_n(x;a,1,0|q) &=\lim_{b\to0} J_n(x;a/b,b,1,0|q)\;. \nonumber
\end{align}
\end{remark}

\begin{remark}\label{remarkLJ2}
The monic $q$-classical OPS (see Table \ref{table1}), up to affine transformations of the variable,
can be obtained from the OPS given in {\bf (I)} and {\bf (II)}.  
Indeed,
\begin{align}
U_n^{(a)}(x) & =L_n(x;a,1,0|q)  \nonumber \\
L_n(x;a,b|q) & =(abq)^n L_n\big(x/(abq);1/a,1/b,1|q\big)   \nonumber \\
L_n(x;a|q) & = L_n\big(x;0,1,a|q\big)  \nonumber  \\
l_n(x;a|q) & = L_n\big(x;0,0,-a|q\big)  \nonumber  \\
P_n(x;a,b,c|q) & =\left\{
\begin{array}{lcl}
q^{n}J_n(x/q;1,a,c,ab|q)\;, & \mbox{\rm if} & b\neq0 \\ [0.25em]
(acq)^{n}L_n(x/(acq);1/a,1/c,1|q)\;, & \mbox{\rm if} &  b=0
\end{array}
\right.  \nonumber \\
P_n(x;a,b|q) & =\left\{
\begin{array}{lcl}
J_n(x;0,a,1,ab|q)\;, & \mbox{\rm if} & b\neq0 \\ [0.25em]
a^nL_n(x/a;1/a,0,1|q)\;, & \mbox{\rm if} &  b=0
\end{array}
\right.   \nonumber \\
B_n(x;a|q) & =J_n(x;0,0,1,-a/q|q\big)  \nonumber \\
j_n(x;a,b|q) & = q^nJ_n\big(x/q; b,0,0, a/q|q\big)\,, \nonumber
\end{align}
where in each case the parameters are subject to the restrictions given in Table \ref{table1}.
\end{remark}

\begin{remark}\label{remarkLJ}
The converse of Remark \ref{remarkLJ2} is also true, that is, the monic OPS in {\bf (I)} and {\bf (II)}
can be obtained from the monic $q$-classical OPS. Indeed,
\smallskip

{\rm (i)} 
If $c=0$ in the definition of $L_n(\cdot;a,b,c|q)$ (and so $ab\neq0$, by the regularity conditions), we obtain (monic) Al-Salam-Carlitz polynomials:
 $$
L_n(x;a,b,0|q)=b^nU_n^{(a/b)}\left(x/b|q\right)\,.
$$
Consider now $c\neq0$. If $ab\neq0$, we obtain Big $q$-Laguerre polynomials: 
$$
L_n(x;a,b,c|q)=\left(ab/(cq)\right)^nL_n\left(cqx/(ab);c/a,c/b\big| q\right) \,;
$$
if $ab=0$ and $|a|+|b|\neq0$, we obtain Little $q$-Laguerre polynomials: 
$$
 L_n\big(x;a,b,c| q\big)=
\left\{
\begin{array}{rcl}
b^{n}L_n(x/b;c/b|q) & \mbox{\rm if} & \mbox{\rm $a=0$ and $b\neq0$} \,,\\ [0.25em]
a^{n}L_n(x/a;c/a|q) & \mbox{\rm if} &  \mbox{\rm $a\neq0$ and $b=0$}\,;
\end{array}
\right.
$$
and if $a=b=0$, we obtain one of the monic OPS given by Medem and \'Alvarez-Nodarse in \cite[Table 2]{MA-N2001}:
$$L_n (x;0,0,c|q)=l_n(x;-c|q)\,.$$

{\rm (ii)}
If $d=0$ in the definition of $J_n(\cdot;a,b,c,d|q)$, the regularity conditions imply $bc\neq0$, and we obtain
Little $q$-Laguerre polynomials if $a=0$ and Big $q$-Laguerre polynomials if $a\neq0$, according to {\rm (i)} and the relation 
$$ J_n(x;a,b,c,0|q)=L_n\big(x;ab,c,bc| q\big)\;.$$ 
Consider now $d\neq0$. If $abc \neq 0$, we obtain Big $q$-Jacobi polynomials:
$$
J_n(x;a,b,c,d|q)=(a/q)^nP_n(qx/a;b,d/b,c/a|q)\;;
$$
if only one among $a$, $b$, and $c$ is zero, then we obtain Little $q$-Jacobi polynomials:
$$
 J_n\big(x;a,b,c,d| q\big)=
\left\{
\begin{array}{lcl}
c^nP_n(x/c;b,d/b|q) & \mbox{\rm if} & \mbox{\rm $a=0$ and $bc\neq0$} \,,\\ [0.25em]
c^nP_n(x/c;ad/c,c/a|q) & \mbox{\rm if} &  \mbox{\rm $b=0$ and $ac\neq0$}\,,\\ [0.25em]
(ab)^nP_n(x/(ab);d/b,b|q) & \mbox{\rm if} &  \mbox{\rm $c=0$ and $ab\neq0$}\,;
\end{array}
\right.
$$
if $a=b=0$ (and so $c\neq0$, by regularity), we obtain $q$-Bessel polynomials:
$$
 J_n\big(x;0,0,c,d| q\big)=c^nB_n(x/c;-dq|q)\;;
 $$
and if $b=c=0$ (and so $a\neq0$, by regularity) we obtain the other monic OPS given by Medem and \'Alvarez-Nodarse in \cite[Table 2]{MA-N2001}:
$$
J_n (x;a,0,0,d|q)=q^{-n}j_n(qx;qd,a|q)\;.
$$
(There are no additional cases, since the condition $d\neq0$ together with the regularity conditions for $(J_n (\cdot;a,b,c,d|q))_{n\geq 0}$ imply $(a,c)\neq(0,0)$.)
\end{remark}

\begin{remark}\label{final}
Note that the $q$-classical OPS are (up to affine changes of the variable) special or limiting cases of the polynomials $J_n$. 
\end{remark}

To achieve the required characterization of \eqref{Griffin1}, we also need to use the following fact \cite[Theorem 1.2]{Nossa}:
if $\textbf{u}\in\mathcal{P}^*$ is a regular linear functional satisfying
$${\bf D}_{1/q,-\omega/q} (\phi {\bf u})=\psi {\bf u}\,,$$
where $\phi$ and $\psi$ are nonzero polynomials such that $\deg\phi\leq2$ and $\deg\psi=1$, and
setting $\phi (x)=a_0+a_1x+a_2 x^2$, $\psi(x)=b_0+b_1x$, and
$$d_n :=b_1 q^{-n} +a_2 [n]_{q^{-1}}\;,\quad e_n :=b_0 q^{-n} +(a_1-q^{-1}\omega d_n)[n]_{q^{-1}}$$
for each $n=0,1,\ldots$, then the (regularity) conditions $d_{n}\neq0$ and $\phi(-e_n/d_{2n})\neq0$ hold for each $n=0,1,\ldots$ 
and the coefficients of the three-term recurrence relation for the associated monic OPS are given by
\begin{align}
\beta_n & =-q^{-1}\omega[n]_{q^{-1}}+\frac{[n]_{q^{-1}}e_{n-1}}{d_{2n-2}}-\frac{[n+1]_{q^{-1}} e_n}{d_{2n}}\,, \label{coeffDq2}\\
\gamma_{n+1} & =-\frac{q^{-n}[n+1]_{q^{-1}}d_{n-1}}{d_{2n-1}d_{2n+1}}\phi\left(-\frac{e_n}{d_{2n}} \right)\,,\quad n=0,1,2,\ldots. \label{coeffDq1}
\end{align}

\begin{theorem}\label{DGThm}
A monic OPS $(P_n)_{n\geq0}$ satisfies \eqref{Griffin1} if and only if, up to an affine transformation of the variable, it is a $q$-classical monic OPS.
\end{theorem}

\begin{proof}
In the analysis of the structure relation \eqref{Griffin1} we consider the three possible cases,
according to the degree of the (monic) polynomial $\pi_N$, $N\in\{0,1,2\}$.
\medskip

{\sc case i}: $N=0$. Then $\pi_0(x)=1$ and so (\ref{Griffin1}) becomes
$$D_{q,\omega} P_{n+1}(x)=[n+1]_q P_n(x)\;,\quad n=0,1,\ldots\;.$$
From \eqref{EqFuc1}, \eqref{EqFuc2}, and \eqref{TTRRPnCo}, we see that ${\bf u}$ satisfies the functional equation
\begin{align*}
{\bf D}_{1/q,-\omega/q} {\bf u}=-\frac{q}{\gamma_1}(x-\beta_0){\bf u}\,.
\end{align*}
Let $a$ and $b$ be the two roots of the quadratic equation
$$
z^2+(\omega_0-\beta_0)z+\gamma_1/(q-1)=0\,.
$$
Note that $ab\neq0$, $\gamma_1 =ab(q-1)$, and $\beta_0 =a+b+\omega_0$, where
$$\omega_0 :=\frac{\omega}{1-q}\,.$$
Using (\ref{coeffDq2}) and (\ref{coeffDq1}), the recurrence coefficients for the monic OPS $\big(P_n\big)_{n\geq0}$ are 
$$
\beta_n=\omega_0+(a+b)q^n \;,\quad \gamma_{n+1}=-ab(1-q^{n+1})q^n \;,\quad n=0,1,\ldots\,.
$$
This means that
\begin{align*}
P_n (x)=L_n \left( x-\omega_0;a,b,0|q\right)\;,\quad n=0,1,\ldots\, .
\end{align*}
(Thus, according to Remark \ref{remarkLJ2}, in this case, up to affine transformations of the variable, we obtain Al-Salam-Carlitz polynomials.)
\medskip

{\sc case ii}: $N=1$. Writing $\pi_1(x)=x-\omega_0+c$, $c\in\mathbb{C}$, (\ref{Griffin1}) becomes
$$(x-\omega_0+c)P_n^{[1]}(x)=P_{n+1}(x)+c_{n,n} P_n(x)\;,\quad n=0,1,\ldots\;.$$
Setting $n=0$ gives $c_{0,0}=c+\beta_0-\omega_0$, and so condition \eqref{cond1} implies $c+\beta_0\neq\omega_0$.
By \eqref{EqFuc1}, \eqref{EqFuc2}, and \eqref{TTRRPnCo}, we obtain the functional equation
\begin{align*}
{\bf D}_{1/q,-\omega/q}\big((x-\omega_0+c) {\bf u}\big)=-\frac{q(\beta_0-\omega_0+c)}{\gamma_1}(x-\beta_0){\bf u}\,.
\end{align*}
Setting $\alpha=-q(\beta_0-\omega_0+c)/\gamma_1$ (hence $\alpha\neq0$) and $\beta=\alpha(\omega_0-\beta_0)$,
the above functional equation becomes
$${\bf D}_{1/q,-\omega/q}\big((x-\omega_0+c) {\bf u}\big)=\big(\alpha (x-\omega_0)+\beta\big){\bf u}\,,$$
hence, using  (\ref{coeffDq2}) and (\ref{coeffDq1}), we find
\begin{align}
\beta_n &=\omega_0 -\frac{\big(\beta(1-q)+(1+q)(1-q^n)\big)q^{n}}{\alpha (1-q)}\,, \label{betanQQ} \\
\gamma_{n+1} &=\frac{(1-q^{n+1})q^{n+1}\big(\alpha c(1-q)+(q+\beta(1-q))q^n-q^{2n+1}\big)}{\alpha^2 (1-q)^2} \label{gammanQQ}
\end{align}
for each $n=0,1,\ldots$.
Let $a$ and $b$ be the zeros of the polynomial
$$
\theta_2(z):=-q^{-1}z^2+\big(1-\beta(1-q^{-1})\big)z+\alpha c(1-q)\;.
$$
Then $a+b=q+\beta(1-q)$ and $ab=c\alpha q(1-q)$.
Therefore, setting $r:=1/(\alpha(q-1))$,
we have $r\neq0$ and, from \eqref{betanQQ} and \eqref{gammanQQ}, 
\begin{align*}
\beta_n &=\omega_0+rq^{n}\big(a+b+1-q^n-q^{n+1}\big)\,,\\
\gamma_{n+1}&=-r^2q^n\big(1-q^{n+1}\big)\big(a-q^{n+1}\big)\big(b-q^{n+1}\big)
\end{align*}
for each $n=0,1,\ldots$. This means that
\begin{align*}
P_n (x)= r^n L_n \left( \frac{x-\omega_0}{r};a,b,1\Big|q \right) =L_n \left( x-\omega_0;ar,br,r \Big|q \right)\,.
\end{align*}
(Therefore, in this case, up to affine transformations of the variable, we obtain Big-$q$-Laguerre polynomials if $ab\neq0$, Little-$q$-Laguerre polynomials if $ab=0$ and $a$ and $b$ do not vanish simultaneously, and the OPS $(l_n)_{n\geq0}$ if $a=b=0$.)
\medskip

{\sc case iii}: $N=2$. Then we may write $\pi_2(x)=(x-\omega_0 -r)(x-\omega_0 -s)$,  
with $r,s\in\mathbb{C}$, and (\ref{Griffin1}) becomes
$$(x-\omega_0 -r)(x-\omega_0 -s)P_n^{[1]}(x)=P_{n+2}(x)+c_{n,n+1} P_{n+1}(x)+c_{n,n} P_n(x)$$
for each $n=0,1,\ldots$. From (\ref{EqFuc1}) and (\ref{EqFuc2}), we deduce
$${\bf D}_{1/q,-\omega/q} \big((x-\omega_0 -r)(x-\omega_0 -s){\bf u}\big)=(\alpha (x-\omega_0 )+\beta){\bf u} \,,$$
where $\alpha := -q\big(\gamma_1+\pi_2(\beta_0)\big)/\gamma_1$ and $\beta =-\alpha (\beta_0 -\omega_0)$.
The regularity of ${\bf u}$ implies $\alpha\neq0$. Since $(1-q^{-1})d_n = 1+ (-1+\alpha (1-q^{-1}))q^{-n} $, then we will distinguish two sub-cases, depending whether $(d_n)_ {n\geq0}$ is a constant sequence or not.

{\sc case iii}.a) If $\alpha = 1/(1-q^{-1})$, then $d_n =\alpha$ for all $n$. Let $c:=(q-1)\beta+q(r+s)$. 
By using (\ref{coeffDq2})--(\ref{coeffDq1}) we find $$P_n (x) =L_n \left( x-\omega_0;r,s,c |q^{-1}\right)\;, \quad n=0,1,\ldots.$$
(Therefore, in this case, we obtain Al-Salam-Carlitz polynomials if $c=0$, i.e., $\beta_0=\omega_0+r+s$, Big-$q$-Laguerre polynomials if $rs\neq0$, Little-$q$-Laguerre polynomials if $rs=0$ and $r$ and $s$ do not vanish simultaneously, and the OPS $(l_n)_{n\geq0}$ if $r=s=0$.) 

{\sc case iii}.b) If $\alpha \neq 1/(1-q^{-1})$, then $(d_n)_ {n\geq0}$ is not a constant sequence. 
Let $$u:=q(q+\alpha(1-q))\;,\quad \lambda :=(r+s)q-\beta (1-q)\,.$$
Note that $u\neq0$ (since $\alpha\neq 1/(1-q^{-1})$). Note also that 
$d_n=(1-uq^{-n-2})/(1-q^{-1})$ for each $n=0,1,\ldots$ and so, since $d_n\neq0$, we obtain $u\neq q^n$ for each $n=0,1,\ldots$.
Therefore, using (\ref{coeffDq2})--(\ref{coeffDq1}), we obtain
\begin{align}
\beta_n =  \omega_0 + q^{-n} \frac{(\lambda+r+s)(1+uq^{-2n-1})-(1+q^{-1})(\lambda +ru+su)q^{-n}}{(1-uq^{-2n})(1-uq^{-2n-2})}  \label{coefbeta} 
\end{align}
and
\begin{equation}\label{coefgamma} 
\gamma_{n+1} =  -\frac{q^{-n}(1-q^{-n-1})(1-uq^{-n-1})\varphi(q^{-n-1};r,s)\varphi(q^{-n-1};s,r)}{(1-uq^{-2n-1})(1-uq^{-2n-2})^2 (1-uq^{-2n-3})}
\end{equation}
for each $n=0,1,\ldots$, where
$$
\varphi(z;x,y):=xu z^2 -\lambda z +y\,.
$$
If $r=\lambda=0$ then $\varphi(z;r,s)=s$ and $\varphi(z;s,r)=suz^2$. Then, from (\ref{coefbeta})--(\ref{coefgamma})
we see that this case may occur only if $s\neq0$ and 
$$P_n (x)= J_n \left( x- \omega_0;0,0,s,u|q^{-1}\right)\;,\quad n=0,1,\ldots.$$
(This means that, in this case, the $P_n$'s are $q$-Bessel polynomials.)
If $r=0$ and $\lambda\neq0$, define $a=\lambda/u$ and $b=us/\lambda$;
and if $r\neq0$ ($\lambda$ being zero or not),  define
$a=(\lambda+\sqrt{\Delta})/(2u)$ and $b=(\lambda-\sqrt{\Delta})/(2r)$, where $\Delta:=\lambda^2-4rsu$
(alternatively, we may choose $a=(\lambda-\sqrt{\Delta})/(2u)$ and $b=(\lambda+\sqrt{\Delta})/(2r)$).
These choices of $a$ and $b$ (in either cases $r=0$ and $\lambda\neq0$, or $r\neq0$) give $s=ab$ and $\lambda =au+br$, and so
$$
\varphi(z;r,s)=(rz-a)(u z-b)\;,\quad  \varphi(z;s,r)=(auz-r)(b z-1)\,.
$$
Therefore, using (\ref{coefbeta})--(\ref{coefgamma}), we obtain
$$P_n (x)= J_n \left( x- \omega_0;a,b,r,u|q^{-1}\right)\;,\quad n=0,1,\ldots.$$
(In this case, if $r=0$ (and so $a\neq0$) we obtain Little-$q$-Jacobi polynomials if $b\neq0$ and the OPS $(j_n)_n$ if $b=0$; and if $r\neq0$, we obtain Big-$q$-Jacobi polynomials if $a,b\neq0$, Little-$q$-Jacobi polynomials if $ab=0$ and $a$ and $b$ do not vanish simultaneously, and $q-$Bessel polynomials if $a=b=0$.) 
\end{proof}

\section*{Final remarks}

This paper highlights that the concept of coherent pair of measures, besides its own theoretical interest, is a useful tool to deal with specific algebraic problems in the theory of orthogonal polynomials. This is shown in Section 3 (Theorem 3.1) by proving the $(q,\omega)$-analogue of the Al-Salam and Chihara characterization of the classical OPS \cite{ASC1972}. Moreover, Theorem 1.3 when $\omega=0$ corrects a miss-stated result from \cite{G2016} where two additional families of monic OPS were missed, namely the polynomials $l_n(\cdot;a|q)$ and $j_n(\cdot;a,b|q)$ mentioned in Table \ref{table1}. It is worth mentioning that several problems and conjectures related with this type of structure relations remain unsolved (see \cite[Section 24.7.1]{IsmailBook2005}). For instance, the following conjecture was posed by M. E. H Ismail in his monograph on Orthogonal Polynomials and Special Functions published in 2005 (see \cite[Conjecture 24.7.7 ]{IsmailBook2005}) and revised in 2009:

\begin{conj}\label{conj}
If an OPS $(p_n)_{n\geq0}$ satisfies
$$
\pi(x) D_{q,0} p_n(x)=\sum_{k=-r}^s c_{n, k} p_{n+k}(x),
$$
for some positive integers $r$ and $s$, and a polynomial $\pi(x)$ which does not depend on $n$, then we have the orthogonality relation 
$$
\int_a^b p_m(x) p_n(x)w(x)\mathrm{d}_q x=0, \quad m\not=n, \quad m=0,1,\ldots,
$$
$w$ being a function which is q-integrable on the interval $(a, b)$ and which satisfy the equation 
$$
D_{q,0} w(x)=-u(q x) w(q x),
$$
where $u$ is a rational function.
\end{conj}

Note that the results of the previous section (in the case $\omega=0$) give more that a simple answer to a particular case of the above problem. Indeed, by Theorem \ref{DGThm}, we known that, up to affine transformations of variable, the only monic OPS satisfying the hypotheses of Conjecture \ref{conj}, with $\pi(x)$ being a polynomial of degree at most $2$, $r=s$, and $c_{n, -r}\not=0$, are the monic $q$-classical OPS given in Table \ref{table1}, which, according to Remark \ref{remarkLJ2}, \ref{remarkLJ} and \ref{final}, are special or limiting cases of a unique family, $(J_n (x;a,b,c,d|q))_{n\geq 0}$. This gives positive answer to Conjecture \ref{conj} under our restrictions, after rewrite the Pearson equation for the $q$-classical OPS.

\section*{Acknowledgements }
This work is supported by the Centre for Mathematics of the University of Coimbra - UIDB/00324/2020, funded by the Portuguese Government through FCT/ MCTES. RAN was partially supported by PGC2018-096504-B-C31 (FEDER(EU)/ Ministerio de Ciencia e Innovaci\'on-Agencia Estatal de Investigaci\'on), FQM-262 and Feder-US-1254600 (FEDET(EU)/Jun\-ta de Anda\-lu\-c\'ia).
DM is partially supported by the Centre for Mathematics of the University of Coimbra, funded by the Portuguese Government through FCT grant PD/BD/135295/2017.

\end{document}